\documentclass[12pt,english]{amsart}

\usepackage[left=2.5cm,top=3cm,right=2.5cm,bottom=3cm]{geometry}

\usepackage{color}
\usepackage[pdfstartpage=1,colorlinks=true,bookmarks=false,pdfstartview={FitH}]{hyperref}
\usepackage{soul}
\usepackage{cite}
\usepackage{amsmath,amssymb}
\usepackage{mathtools}
\usepackage{verbatim}
\usepackage{comment}
\usepackage{esint}
\usepackage{cleveref}
\DeclarePairedDelimiter\abs{\lvert}{\rvert}
\makeatletter
\let\oldabs\abs
\def\abs{\@ifstar{\oldabs}{\oldabs*}}
\makeatletter
\newcommand{\vast}{\bBigg@{4}}
\newcommand{\Vast}{\bBigg@{5}}
\makeatother

\newcommand{\eps}{\varepsilon}

\newcommand{\R}{\mathbb{R}}

\newcommand{\bS}{\mathbb{S}}
\newcommand{\p}{\partial}

\newcommand{\Ds}{(-\Delta)^{s}}

\newcommand{\norm}[2][]{\left\|{#2}\right\|_{#1}}

\newcommand{\set}[1]{\left\{#1\right\}}

\newcommand{\textas}{\text{ as }}
\newcommand{\texton}{\text{ on }}

\newcommand{\textin}{\text{ in }}

\newcommand{\textwith}{\text{ with }}
\newcommand{\textfor}{\text{ for }}

\newcommand{\textand}{\text{ and }}

\newcommand{\oneset}[1]{\mathbf{1}_{\set{#1}}}

\newcommand{\pnu}[1]{\dfrac{\partial{#1}}{\partial\nu}}
\newcommand{\dist}{{\rm dist}\, }

\newcommand{\cE}{\mathcal{E}}

\newcommand{\cN}{\mathcal{N}}
\newcommand{\cG}{\mathcal{G}}

\newcommand{\cL}{\mathcal{L}}

\newcommand{\cT}{\mathcal{T}}
\newcommand{\cX}{\mathcal{X}}

\newcommand{\Dh}{(-\Delta)^{\frac12}}

\theoremstyle{plain}
\newtheorem{thm}{Theorem}[section]
\newtheorem{lem}[thm]{Lemma}
\newtheorem{cor}[thm]{Corollary}
\newtheorem{prop}[thm]{Proposition}

\newtheorem*{prop*}{Proposition}

\theoremstyle{definition}
\newtheorem{defn}[thm]{Definition}

\theoremstyle{remark}
\newtheorem{remark}[thm]{Remark}
\newcommand{\bremark}{\begin{remark} \em}
\newcommand{\eremark}{\end{remark} }

\numberwithin{equation}{section}

\definecolor{g1}{rgb}{0,0.5,0.1}
\definecolor{g2}{rgb}{0,0.6,0}
\definecolor{r2}{rgb}{0.8,0,0}

\vbadness=\maxdimen
\hbadness=\maxdimen

\begin{document}

\title[An analytic construction]
    {
    An analytic construction of singular solutions related to a critical Yamabe problem
    }


\author{Hardy Chan}
\email[H.~Chan]{hardy.chan@math.ethz.ch}
\address[H.~Chan]{Department of Mathematics, ETH Z\"{u}rich}

\author{Azahara DelaTorre}
\email[A. DelaTorre]{azahara.de.la.torre@math.uni-freiburg.de}
\address[A. DelaTorre]{Mathematisches Institut, Albert-Ludwigs-Universit\"{a}t Freiburg}

\begin{abstract}
We answer affirmatively a question of Aviles posed in 1983, concerning the construction of singular solutions of semilinear equations without using phase-plane analysis. Fully exploiting the semilinearity and the stability of the linearized operator in any dimension, our techniques involve a careful gluing in weighted $L^\infty$ spaces that handles multiple occurrences of
criticality, without the need of derivative estimates. 

The above solution constitutes an \emph{Ansatz} for the Yamabe problem with a prescribed singular set of maximal dimension $(n-2)/2$, for which, using the same machinery, we provide an alternative construction to the one given by Pacard. His linear theory uses $L^p$-theory 
on manifolds, while our approach studies the equations in the ambient space and is
therefore suitable for generalization to nonlocal problems. In a forthcoming paper, we will
prove analogous results in the fractional setting.

\end{abstract}

\maketitle


\section{Introduction}

We are concerned with the construction of singular solutions of the semilinear elliptic equation with superlinear nonlinearity,
\begin{equation}\label{eq:main2}
{-\Delta} u=u^{\frac{N}{N-2}}
    \quad \textin B_1\setminus\set{0},
\end{equation}
in which the exponent $N/(N-2)$ is critical for the existence of singular solutions, below which the singularity is removable (see for example \cite[Proposition 3.5]{QSbook}). Throughout the paper, we assume that $N\geq 3$, and, because of the singularity, solutions are understood in the very weak sense. For \eqref{eq:main2}, we say that $u\in L^{\frac{N}{N-2}}(B_1)$ is a solution if
\[
\int_{B_1}
    -u\Delta\zeta
\,dx
=\int_{B_1}
    u^{\frac{N}{N-2}}\zeta
\,dx
+\int_{\p B_1}
    u\pnu{\zeta}
\,d\sigma,
    \quad \forall \zeta\in C^2(\overline{B_1})
    \textwith \zeta|_{\p B_1}=0.
\]

\subsection{Singular solutions
}
In a series of papers by Aviles \cite{Aviles-1,Aviles-2}, he provided, in particular, the behaviour of singular solutions for \eqref{eq:main2} showing that
\[
u(x)=
	\left(
		\dfrac{N-2}{\sqrt{2}}
	\right)^{N-2}
	(1+o(1))
	\dfrac{1}{
    	r^{N-2}
    	(\log\frac{1}{r})^{\frac{N-2}{2}
	}
},
\]
as $r:= |x| \searrow 0$. The author constructed radial solutions using ODE analysis and asked for a more analytic construction. Here we give a positive answer using the gluing method. In particular, we will prove the following result:

\begin{thm}[Existence of singular solutions]
\label{thm:2}
There exists $\bar{\eps}\in(0,e^{-1})$ 
such that for any $\eps\in(0,\bar{\eps}]$, there exists a smooth positive radial solution $\bar{u}$ of
\eqref{eq:main2}
such that
\begin{equation}\label{eq:u-bar-asymp}
\bar{u}(r)=
	\left(
		\dfrac{N-2}{\sqrt{2}}
	\right)^{N-2}
	\left(
		1-\frac{N}{4}
		\dfrac{
			\log\log\frac{1}{\eps r}
		}{
			\log\frac{1}{\eps r}
		}
		+O\left(
			\dfrac{1}{
				(\log\frac{1}{\eps r})^{\frac32}
			}
		\right)
	\right)
	\dfrac{1}{
    	r^{N-2}
    	(\log\frac{1}{\eps r})^{\frac{N-2}{2}
	}
},
\end{equation}
as $r\searrow 0$.
\end{thm}

\begin{remark}
The error $(\log\frac{1}{\eps r})^{-\frac32}$ is not optimal, but simply fixed for the simplicity of presentation. From the proof we see that the exact behavior has the power $(\log\frac{1}{\eps r})^{-2}$ up to a $\log\log$-correction, and more precise expansions are explicitly computable.
\end{remark}

It is a standard approach to consider H\"{o}lder spaces in gluing constructions, for the control
of derivatives and the bijectivity of the differential operators in view of Schauder estimates. When the growth are of power type, a weighted H\"{o}lder space is a natural space to work on. In the present situation, however, polylogarithmic weights appear all over. Our idea is the sole use of weighted $L^\infty$ spaces, thus avoiding unnecessary heavy computations as one would expect with a weighted H\"{o}lder space.

The actual weights involved are $\log$-polyhomogeneous in nature, as one may see in \eqref{eq:u-bar-asymp}. This is due to the criticality of the problem and is elaborated with the introduction of the precise setting in \Cref{sec:rad-crit}.

The proof is robust and applies to much more general equations, as long as the first approximation is \emph{stable}, i.e. the linearized operator is positive in the sense that the associated quadratic form is non-negative definite. The use of stability is known to experts; an example of this is the invertibility of the Jacobi operator when the right hand side has very fast decay, see \cite[Proposition 4.2]{DKW}. We observe that  For \eqref{eq:main2}, the \emph{Ansatz} $u_1=c_0r^{-N}(\log\frac{1}{\eps r})^{-(N-2)/2}$ is stable in \emph{all} dimensions\footnote{In a neighborhood of the origin, but one may use the scaling invariance to write down a stable solution in $B_1\setminus\set{0}$.}, because the linearized operator
\[
-\Delta-\frac{N}{N-2}u_1^{\frac{2}{N-2}}
=-\Delta-\frac{N(N-2)}{2}\frac{1}{r^2\log\frac{1}{\eps r}}
\]
is positive by Hardy inequality, in view of the helping logarithmic correction.

The recent striking regularity result of Cabr\'{e}, Figalli, Ros-Oton and Serra \cite{CFRS} gives another reason, besides the direct verification that $u_1\notin H^1$, why such singular stable solutions have to be understood in a sense weaker than the variational one, at least in dimensions $N\leq 9$.

The solution in question represents the building block for constructing solutions to the singular Yamabe problem, as we discuss below.

\subsection{Singular Yamabe problem}

Given a compact Riemannian manifold $(M^n,g)$ the Yamabe problem asks for a conformal metric $g_u=u^{\frac{4}{n-2}}g$ with constant scalar curvature. In the case of a sphere, the equivalent (via the stereographic projection) PDE formulation is a semilinear equation with a Sobolev critical exponent,
\[
-\Delta u = u^{\frac{n+2}{n-2}}
	\quad \textin \R^n.
\]
The combined work of Trudinger, Aubin and Schoen provided a complete solution to this problem in 1984 (see e.g. \cite{LP} and the references therein). In consequence, it is particularly interesting to study the problem in a curved setting, in particular, in Euclidean space when we allow the presence of singularities. In the singular Yamabe problem one looks for solutions which are singular on some set $\Sigma$. By a theorem of Schoen and Yau \cite{SY}, if $g_u$ is complete then $\Sigma$ is at most $(n-2)/2$-dimensional. Such singular solutions are indeed constructed by Pacard \cite{Pacard-1} and Mazzeo and Pacard \cite{MP,MP2}, where the authors provide solutions which are singular (exactly) on a $k$-dimensional submanifold with $k=\frac{n-2}{2}$ (for $n\geq 4$ even) and $k\in[0,\frac{n-2}{2})$, respectively.\footnote{with the case $k=0$ corresponding to isolated singularities} {Note that solutions with isolated singularities were already constructed by Schoen in \cite{Schoen_iso} and, indeed, \cite{MP2} presents a simplification of his long but remarkable proof.}

With respect to the (lower) codimension $N:=n-k\in[\frac{n+2}{2},n]$, the exponent $p=\frac{n+2}{n-2}=\frac{N+k+2}{N+k-2}$ is Sobolev subcritical. More explicitly, one is led to study singular solutions of
\[
-\Delta u=u^p
	\quad \textin \R^N\setminus\set{0}
\]
for $p=\frac{N+k+2}{N+k-2}<\frac{N+2}{N-2}$.
In this regime, it is known that the \emph{fast-decaying} radial solution $\bar{u}(r)$ of $-\Delta \bar{u}=\bar{u}^p$ in $\R^N\setminus\set{0}$ exists as a building block, meaning that\footnote{Hereafter $f \asymp g$ means $f$ and $g$ are bounded by a positive multiple of each other.}
\[
\bar{u}(r)\asymp
\begin{cases}
r^{-\frac{2}{p-1}}
	& \textas r\searrow 0,\\
r^{-(N-2)}
	& \textas r\nearrow \infty,
\end{cases}\]
{provided that} $N-2>2/(p-1)$, which is equivalent to $p>N/(N-2)$. When $p=N/(N-2)$, the scaling-invariant power $2/(p-1)=N-2$ corresponds to the fundamental solution, so a logarithmic correction must be inserted so that the nonlinear equation is satisfied. This \emph{slow-growing} behavior was found by Aviles \cite{Aviles-1,Aviles-2} using ODE arguments to be
\[
\bar{u}(r)\asymp r^{-(N-2)}
	\left(\log\frac1r\right)^{-\frac{N-2}{2}}
	\quad \textas r\searrow 0.
\]
In both cases, a smallness can be obtained by rescaling\footnote{More precisely, blow-up for $p=N/(N-2)$ and blow-down for $p>N/(N-2)$.} $\bar{u}(r)$, and this is crucially used in \cite{Pacard-1,MP}. 


When the conformally related metric is not necessarily complete, Pacard \cite{Pacard-2} constructed singular solutions for $n=4,6$ such that the singular set may have any Hausdorff dimension in the inverval $[\frac{n-2}{2},n]$. For dimension $n\geq 9$, Chen and Lin \cite{CL} constructed weak solutions singular in the whole $\R^n$. Both constructions are variational and use the \emph{stability} of the radial solution, meaning that the quadratic form associated to the linearized operator around $\bar{u}$ is non-negative definite, i.e.
\begin{equation}\label{eq:stability}
\int |\nabla\zeta|^2-p\bar{u}^{p-1}\zeta^2\geq 0,
\end{equation}
for smooth test functions $\zeta$ with compact support. This is true only when
\[
\frac{N}{N-2}\leq p<p_1:=1+\frac{4}{N-4+2\sqrt{N-1}}.
\]
Note that the threshold of stability satisfies $p_1<\frac{N+2}{N-2}$.

Recently, the fractional curvature, a non-local intrinsic concept defined from the conformal fractional Laplacian, has caught important attention in problems arising in conformal geometry, and a parallel study to the local one, is being developed for this problem. The fractional Yamabe problem arises when we try to find a metric conformal to a given one and which has constant fractional curvature, and it is equivalent to look for solutions of
\[
\Ds u=u^\frac{n+2s}{n-2s}
	\quad \textin \R^n,
\]
where $s\in(0,1)$.
Note that we restrict ourselves to the case $s\in(0,1)$, in order that the extra difficulties that we are dealing with come from the nonlocality, and not from the loss of maximum principle. Thus, the singular fractional Yamabe problem is
\[
\Ds u=u^\frac{n+2s}{n-2s}
	\quad \textin \R^n\setminus\Sigma,
\]
where $\Sigma$ is a singular set of dimension $k$ satisfying
\[
\left.\Gamma\left(\frac{n-2k+2s}{4}\right)
\right/
\Gamma\left(\frac{n-2k-2s}{4}\right)
\geq 0,
\]
which is true in particular when $k\in[0,\frac{n-2s}{2}]$. This dimension restriction is due to Gonz\'{a}lez, Mazzeo and Sire \cite{GMS}. Again, it is customary to consider the model problem on the normal space with isolated singularity,
\[
\Ds u=u^p
    \quad \textin \R^N\setminus\set{0},
\]
in dimension $N=n-k$ with $p=\frac{N+k+2s}{N+k-2s}<\frac{N+2s}{N-2s}$.

The case $k=0$ has been studied in a series of papers by DelaTorre and Gonz\'{a}lez \cite{DG}, DelaTorre, del Pino, Gonz\'{a}lez and Wei \cite{DPGW} and Ao, DelaTorre, Gonz\'{a}lez and Wei \cite{ADGW}.
In the stable case when $\frac{n}{n-2s}<p<p_1(s)$, where $p_1(s)$ is a suitable threshold exponent corresponding to $p_1$, Ao, Chan, Gonz\'{a}lez and Wei \cite{ACGW} generalized the result of \cite{CL}, where the fast-decay solution for the nonlocal ODE comes from the extremal solution of an auxiliary problem.

A recent paper of Ao, Chan, DelaTorre, Fontelos, Gonz\'{a}lez and Wei \cite{fat} extends the result of \cite{MP}, which covers not only the stable regime $\frac{n}{n-2s}<p<p_1(s)$ but also the unstable one, i.e. $\frac{n}{n-2s}<p<\frac{n+2s}{n-2s}$, thus completing the study of existence when $k\in[0,\frac{n-2s}{2})$. This is done by constructing a fast-decay solution and developing a theory of nonlocal ODE in the spirit of the Frobenius method, using tools from conformal geometry, bifurcation theory, non-Euclidean Fourier analysis and complex analysis. 
See \cite{fat-survey} for an exposition and also \cite{ACGW2} for a related application.

We remark that the case $k=\frac{n-2s}{2}$, corresponding to $p=\frac{n}{n-2s}$, is not covered, due to the limitations of the techniques used in \cite{fat}. Indeed, homogeneity (as opposed to polyhomogeneity, as it appears extensively in the current paper) is crucial in several places throughout the proof, including the construction of the building block, formulation of the extension problem, and the inversion of fractional Hardy--Schr\"{o}dinger operator. This leaves the remaining case $k=\frac{n-2s}{2}$ as an interesting open problem, which we will solve in a forthcoming paper \cite{CD2}, by constructing singular solutions that are singular on a submanifold of dimension $k=(n-1)/2$, for an odd integer $n\geq3$, in the case $s=1/2$. This, in fact, is the original motivation of the present article.

\bigskip
Coming back to the local case with singularity of critical dimension $k=\frac{n-2}{2}$, by exploiting the stability of the linearized operator associated to the radial singular solution, we provide an alternative proof 
which can be easily generalized to the fractional case. The basic idea of the construction, namely the approximation with a singular radial function composed with the distance to the singularity, stems from \cite{Pacard-1, MP, fat}.

In order to avoid unnecessary technicalities in the presentation,
we only consider the Dirichlet problem in a small tubular neighborhood around the singular set. The exact result reads:

\begin{thm}\label{thm:1}
Let $n\geq 4$ be an even integer, $k=\frac{n-2}{2}$ and $\Sigma^k\subset\R^n$ be a $k-$dimensional smooth submanifold. Let $r_*=r_*(\Sigma)>0$ be a universal constant such that the tubular neighborhood $\cT_{r_*}$ of width $r_*$ around $\Sigma$ is well-defined and satisfies in addition the condition in \Cref{rem:r*}. Then
\begin{equation}\label{eq:main1}
\begin{cases}
-\Delta u=u^{\frac{n+2}{n-2}}
    & \textin \cT_{r_*}\setminus\Sigma,\\
u=0
	& \texton \p\cT_{r_*},
\end{cases}
\end{equation}
has a solution which generates a complete metric for the Yamabe problem. 
Moreover, under the Fermi change of coordinates $\Phi:(0,r_*)\times\bS^{N-1}\times\Sigma\to\cT_{r_*}\setminus\Sigma$ defined in \eqref{eq:Fermi-3},
\[
u(\Phi(r,\omega,y))-\bar{u}(r)
=\begin{cases}
O\left(
	r^{-(N-3)}
\right)
	& \textfor N\geq 4,\\
O\left(
	(\log\frac1r)^{\frac14}
\right)
	& \textfor N=3,
\end{cases}
\]
as $r\searrow 0$, where $\bar{u}(r)$ is the singular radial solution given by \Cref{thm:2}.
\end{thm}



\begin{remark}
As in \Cref{thm:2}, the errors here are not optimal but are sufficient for our purpose, i.e., they are smaller than $\bar{u}(r)$ in a neighborhood of $\Sigma$.
\end{remark}

Our method to prove the two main results, i.e., \Cref{thm:2} and \Cref{thm:1}, is based on an \emph{a priori} estimate using maximum principle with super-solutions, in weighted $L^\infty$ spaces. We stress that it is possible to apply the method of continuity without H\"{o}lder type estimates, since no extra derivatives are involved in the iterations in view of the semilinearity. As mentioned before, this will be robust enough to treat the fractional case
\[
\Dh u=u^{\frac{n+1}{n-1}}
	\quad \textin \R^n\setminus\Sigma,
\]
in our forthcoming paper \cite{CD2}.

The paper will be organized as follows. In \Cref{sec:not} we introduce the notation, functional spaces and some explicit computations that will be used to prove the main results of the paper. \Cref{sec:radial} is dedicated to the construction of a singular solution for \eqref{eq:main2} and it is concluded by proving \Cref{thm:2}. The last \Cref{sec:Yamabe} is devoted to an alternative construction for the Yamabe problem, which is singular along a submanifold of critical dimension. We will follow the same procedure of \Cref{sec:radial} but taking into account the geometry of the singularity. The proof of \Cref{thm:1} will be given at the end of this \Cref{sec:Yamabe}. For the convenience of the reader, we prove a maximum principle in annular regions in \Cref{Ap1}.

\section{Numerology and function spaces}\label{sec:not}

\subsection{The singular radial solution}
\label{sec:rad-crit}

Let us begin with the radial case, observing some occurrences of the criticality of the problem.

First of all, the scaling of \eqref{eq:main2} suggests that the pure power radial solution should behave like $r^{-(N-2)}$. Unfortunately,  since this is the fundamental solution of $-\Delta$, it does not solve our equation. Hence, a correction must be included, and it turns out that the correct factor is  logarithmic one and, in fact, the approximation $u_1(r)$ is of order $r^{-(N-2)}(\log\frac1r)^{-(N-2)/2}$, as observed by Aviles \cite{Aviles-1, Aviles-2}. See \Cref{cor:poly} below.

Moreover, the error produced by $u_1$, which is a multiple of $r^{-N}(\log\frac1r)^{-(N+2)/2}$, is just not enough for the linearized operator around $u_1$, namely $L_1=-\Delta-N(N-2)(2r^2\log\frac1r)^{-1}$, to be inverted. This is because $L_1$ has a kernel that behaves like $r^{-N}(\log\frac1r)^{N/2}$, which is exactly the expected order when the inverse operator is applied to the error. This has two consequences. First, one must improve the logarithmic decay of the error,\footnote{with respect to the blowing-up inverse polynomial} in order to develop a satisfactory linear theory. Second, such error as the inhomogeneity of an ODE requires a further logarithmic correction for the solution, namely $u_2(r)=u_1(r)+c_1r^{-N}(\log\frac1r)^{-N/2}(\log\log\frac1r)$.

%

Motivated by the above discussion, for $\eps\in(0,e^{-1})$ and $r\in(0,1)$, consider the log-polyhomogeneous functions
{\begin{equation}\label{eq:logpoly-def}
\phi_{\mu,\nu,\theta}^\eps
=\dfrac{1}{
    r^{\mu}
    (\log\frac{1}{\eps r})^{\nu}
    (\log\log\frac{1}{\eps r})^{\theta}.
}
\end{equation}}
The parameter $\eps$ is inserted, by exploiting the scaling invariance $u\mapsto \eps^{N-2}u(\eps\cdot)$ of the equation, to make sure the logarithm powers are well-defined and to produce smallness. 
For $\mu,\nu\geq0$, we define the norm in $B_1=\{x\in\R^n; \  |x|<1\}$ by
\begin{equation}\label{eq:norms}
\norm[\mu,\nu]{u}
:=
    \sup_{r\in(0,1)}
        \phi_{\mu,\nu,0}^{\eps}(r)^{-1}
        |u(r)|
\end{equation}
and define the Banach spaces\footnote{Indeed, any Cauchy sequence in a weighted $L^\infty$ space when divided by the weight is a Cauchy sequence in $L^\infty$, whose limit times the weight is the limit of the original sequence.} of functions in $B_{1}$ singular at the origin,
\begin{equation}\label{eq:spaces}
L^\infty_{\mu,\nu}(B_1)=
\set{
    u\in L^1(B_{1}):
    \norm[\mu,\nu]{u}<\infty
}.
\end{equation}
These are functions that blow up at most as fast as the corresponding polyhomogeneity, and is quantitatively small outside the half ball. 

\subsection{Singularity on a submanifold}\label{Sec:sing}
We write the ambient dimension as $n=k+N$, where $k$ and $N$ are respectively the dimensions of the submanifold $\Sigma$ 
and of the normal space $N_y\Sigma$ at any point $y\in\Sigma$. The Fermi coordinates are well-defined on some tubular neighborhood $\cT_{r_*}$ of $\Sigma^k\subset\R^n$ of width $r_*$. In fact, any point $z\in\R^n$ with $\dist(z,\Sigma)<r_*$ can be written as
\begin{equation}\label{eq:Fermi-1}
z=y+\sum_{j=1}^{N}x_j\nu_j(y),
\end{equation}
where $y\in\Sigma^k$ and $(\nu_1(y),\dots,\nu_j(y))$ is a basis for the normal space $N_y\Sigma$ at $y$, and $x=(x_1,\dots,x_N)\in\R^N$ are the coordinates on $N_y\Sigma$. Using polar coordinates in $\R^N$, we set
\begin{equation}\label{eq:Fermi-2}
r=|x|\in[0,r_*)
	\quad \textand \quad
\omega=\frac{x}{|x|}\in\bS^{N-1}.
\end{equation}
Thus \eqref{eq:Fermi-1} and \eqref{eq:Fermi-2} define a diffeomorphism
\begin{equation}\label{eq:Fermi-3}\begin{split}
\Phi:(0,r_*)\times \bS^{N-1} \times \Sigma^k
	& \to 
\cT_{r_*}\setminus\Sigma\subset\R^n \\
\Phi(r,\omega,y)
	& 
=y+\sum_{j=1}^{N}r\omega_j\nu_j(y).
\end{split}\end{equation}
The associated metric $g(r,\omega,y)$ is well-known (see \cite{FmO,MS,MP}), given by
\[\begin{split}
(g_{ij})
&=\begin{pmatrix}
1 & 0 & O(r) \\
0 & r^2 g_{\bS^{N-1},i'j'}(\omega)+O(r^4) & O(r^2) \\
O(r) & O(r^2) & g_{\Sigma,i''j''}(y)+O(r).
\end{pmatrix},
\end{split}\]
where $O(r^\ell),\ \ell=1,2,4$ are uniformly small as $r\searrow 0$, together with all derivatives with respect to the vector fields $r\p_{r}$, $\p_{\omega_{i'}}$, $\p_{y_{i''}}$. (Here $i,j=1,\dots,n$, $i',j'=1,\dots,N-1$, $i'',j''=1,\dots,k$.) This yields the Laplace--Beltrami operator on $(\cT_{r_*}\setminus\Sigma,g)$,
\begin{equation*}
\begin{split}
\Delta_g
&=r^{1-N}\p_r(r^{N-1}\p_r)
	+r^{-2}\Delta_{\omega}
	+\Delta_{y}
	+O(r)\p_{rr}+O(1)\p_{r}
	+\cL_0,
\end{split}
\end{equation*}
as $r\searrow 0$, where $\cL_0$ is a small second order differential operator with at least one derivative in $\omega$ or $y$. In particular, when applied to a function depending only on $r$, we have
\begin{equation}\label{eq:Lap-g-rad}
\begin{split}
\Delta v(r)
&=\Delta_r v
	+O(r)v_{rr}+O(1)v_{r},
\quad
    \Delta_r v=r^{1-N}(r^{N-1}v_r)_r.
\end{split}
\end{equation}

The norms and function spaces defined in \eqref{eq:norms}, \eqref{eq:spaces} concern only the growth in the variable $r$. As a result, in the tubular neighborhood $\cT_{r_*}$ we define similarly
\[
\norm[\mu,\nu]{v}
:=\sup_{\substack{
    r\in(0,r_*)\\
    \omega\in\bS^{N-1}\\
    y\in\Sigma
}}
    \phi_{\mu,\nu,0}^{\eps}(r)^{-1}v(r,\omega,y),
\]
\[
L_{\mu,\nu}^{\infty}(
    (0,r_*)\times \bS^{N-1} \times \Sigma^k
)
:=\set{
    v\in L^1(
        (0,r_*)\times \bS^{N-1} \times \Sigma^k
    ):
    \norm[\mu,\nu]{v}<\infty
},
\]
where $ \phi_{\mu,\nu,0}^{\eps}$ is given in \eqref{eq:logpoly-def}. Note that here we do not need the parameter $\theta$ because, with the \emph{exact} singular solution constructed in \Cref{thm:2}, the error near the singularity $\Sigma$ is only due to its curvature. In other words, by the smoothness of $\Sigma$, the error is as small as $r^{-(N-1)}(\log\frac{1}{r})^{-(N-2)/2}$. Then one may just, for simplicity, forget about the logarithmic decay unless $N=3$, in which case the second anti-derivative of $r^{-2}$ is already logarithmic.


\subsection{Some explicit computations}

We conclude this section with some explicit computations of $\phi_{\mu,\nu,\theta}^{\eps}$ {in the particular case where $\mu=N-2$ is the critical power.
}
Recall that $\phi_{\mu,\nu,\theta}^{\eps}$ is defined in \eqref{eq:logpoly-def}. Morally, the Laplacian of a logarithmically corrected fundamental solution gain two powers in $r$ and one power in $\log\frac1r$.  In fact, we can assert the following

\begin{lem}[Laplacian of log-polyhomogeneous functions]
\label{lem:log-poly}
For any $\nu,\theta\in\R$, $r\in(0,1)$, $\eps\in(0,e^{-1})$,
\[\begin{split}
-\Delta
	\phi_{N-2,\nu,\theta}^{\eps}
&=
    (N-2)\nu\phi_{N,\nu+1,\theta}^\eps
    +(N-2)\theta\phi_{N,\nu+1,\theta+1}^\eps\\
&\quad\;
    -\nu(\nu+1)\phi_{N,\nu+2,\theta}^\eps
    +O\left(\theta\phi_{N,\nu+1,\theta+1}^\eps\right),
\end{split}\]
as $r\searrow 0$.
\end{lem}

\begin{proof}
For simplicity denote $\ell_1=\log\frac{1}{\eps r}$ and $\ell_2=\log\ell_1$, so that $\p_r\ell_1=-r^{-1}$ and $\p_r\ell_2=-r^{-1}\ell_1^{-1}$.
By direct computations,
\[\begin{split}
\p_r\left(
	r^{2-N}\ell_1^{-\nu}\ell_2^{-\theta}
\right)
&=
	(2-N)r^{1-N}\ell_1^{-\nu}\ell_2^{-\theta}
	+\nu r^{1-N}\ell_1^{-\nu-1}\ell_2^{-\theta}
	+\theta r^{1-N}\ell_1^{-\nu-1}\ell_2^{-\theta-1}\\
	-r^{N-1}\p_r(r^{2-N}\ell_1^{-\nu}\ell_2^{-\theta})
&=
	(N-2)\ell_1^{-\nu}\ell_2^{-\theta}
	-\nu \ell_1^{-\nu-1}\ell_2^{-\theta}
	-\theta \ell_1^{-\nu-1}\ell_2^{-\theta-1}\\
\p_r\left(
	-r^{N-1}\p_r(r^{2-N}\ell_1^{-\nu}\ell_2^{-\theta})
\right)
&=
	(N-2)\left(
		\nu r^{-1}\ell_1^{-\nu-1}\ell_2^{-\theta}
		+\theta r^{-1}\ell_1^{-\nu-1}\ell_2^{-\theta-1}
	\right)\\
&\quad\;
	-\nu(\nu+1)\ell_1^{-\nu-2}\ell_2^{-\theta}
	+O(\theta \ell_1^{-\nu-2}\ell_2^{-\theta-1})\\
\end{split}\]
Thus
\[\begin{split}
-\Delta\left(
	r^{2-N}\ell_1^{-\nu}\ell_2^{-\theta}
\right)
&=
	(N-2)\nu r^{-N}\ell_1^{-\nu-1}\ell_2^{-\theta}
	+(N-2)\theta r^{-1}r^{-N}\ell_1^{-\nu-1}\ell_2^{-\theta-1}\\
&\quad\;
	-\nu(\nu+1)r^{-N}\ell_1^{-\nu-2}\ell_2^{-\theta}
	+O(\theta r^{-N}\ell_1^{-\nu-2}\ell_2^{-\theta-1}).
\end{split}\]
\end{proof}

As a special case ($\theta=0$) we note the following
\begin{cor}[Laplacian of polyhomogeneous functions]
\label{cor:poly}
For any $\nu\in\R$, $r\in(0,1)$, $\eps\in(0,e^{-1})$, we have
\[\begin{split}
-\Delta \phi_{N-2,\nu}^{\eps}
&=
    (N-2)\nu\phi_{N,\nu+1}^{\eps}
    -\nu(\nu+1)\phi_{N,\nu+2}^{\eps}.
\end{split}\]
\end{cor}


\section{Construction of a singular radial solution}
\label{sec:radial}

Over this section we will recover the existence results proved, using ODE methods, by Aviles in \cite{Aviles-1}. Here, we will use gluing methods techniques, in its stead.

\subsection{General strategy}

%

Knowing the leading order behavior from \cite{Aviles-1,Aviles-2}, it is tempting to approximate the solution with
\[
u_1^\eps(r)
=c_0\phi_{N-2,\frac{N-2}{2}}^\eps(r)
=\dfrac{c_0}{r^{N-2}(\log\frac{1}{\eps r})^{\frac{N-2}{2}}}.
\]
Unfortunately, the error
\[
-\Delta u_1^\eps-(u_1^\eps)^{\frac{N}{N-2}}
=O\left(
    \dfrac{1}{r^N(\log\frac{1}{\eps r})^{\frac{N+2}{2}}}
\right)
\]
is too large in the sense that the space $L^\infty_{N,\frac{N+2}{2}}$ (defined in \eqref{eq:spaces}) contains the fundamental solution of the linearized operator. As a result, no satisfactory linear theory can be developed there.

As shortly described below, we will consider an approximation of the form
\[
u_2^{\eps}(r)
=
c_0\phi_{N-2,\frac{N-2}{2}}^{\eps}(r)
+c_1\phi_{N-2,\frac{N}{2},-1}^{\eps}(r)
=
	\dfrac{1}{r^{N-2}(\log\frac{1}{\eps r})^{\frac{N-2}{2}}}
	\left(
		c_0+c_1\dfrac{
			\log\log\frac{1}{\eps r}
		}{
			\log\frac{1}{\eps r}
		}
	\right),
\]
extended globally to $u_3^{\eps}(r)$ via a cut-off function, where $c_0$ and $c_1$ are positive constants. This produces an error of the form (\Cref{prop:u3})
\[
E_{3,\eps}
:=
-\Delta u_3^{\eps}-(u_3^{\eps})^{\frac{N}{N-2}}
=O\left(
	\dfrac{
		(\log\log\frac{1}{\eps r})^2
	}{
		r^N(\log\frac{1}{\eps r})^{\frac{N+4}{2}}
	}
\right),
    \quad \textas r\searrow 0,
\]
which, because of the gain in the power of $\log\frac{1}{\eps r}$, has fast enough decay\footnote{with respect to the blowing-up inverse polynomial} for the maximum principle, in the sense that a polyhomogeneous super-solution exists in $L^\infty_{N,\frac{N+3}{2}}$ (\Cref{lem:local-supsol-radial}). 
Thus an \emph{a priori} estimate can be proved and this, together with the method of continuity, sets the cornerstone of the linear theory, namely $L_\eps^{-1}: L^\infty_{N,\frac{N+3}{2}}\to L^\infty_{N-2,\frac{N+1}{2}}$ is a bounded linear operator (\Cref{prop:exist-new}).

We look for a true solution $\bar{u}=u_3^\eps+\varphi$, that solves\footnote{Alternatively, one may consider the equation $-\Delta\bar{u}=\abs{\bar{u}}^{\frac{N}{N-2}}$ and use the maximum principle, as in \Cref{sec:Yamabe}. But this is not necessary.}
\[
-\Delta\bar{u}
=\abs{\bar{u}}^{\frac{2}{N-2}}\bar{u}
\]
As usual the perturbation solves
\[
L_\eps\varphi=-E_{3,\eps}+\cN[\varphi],
\]
where $L_\eps$ is the linearized operator around $u_3^\eps$ and $\cN[\varphi]$ is quadratically small. A standard fixed point argument yields the existence of $\varphi$ in $L^\infty_{N-2,\frac{N+1}{2}}$ (\Cref{prop:G-new}).



Throughout the rest of this section, we will explain every step in details.

\subsection{The approximations}\label{Appr}

We will construct our first approximation based on the the sharp behavior of the solutions provided by Aviles \cite{Aviles-1,Aviles-2}.

\begin{defn}
Let $\eps\in(0,e^{-1})$. Define locally the \emph{first approximation} $u_{1}^{\eps}$ by
\[
u_{1}^{\eps}(r)=c_0\phi_{N-2,\frac{N-2}{2}}^{\eps}(r),
    \quad \textfor r\in(0,1),
\]
with
\[
c_0^{\frac{2}{N-2}}=\frac{(N-2)^2}{2}.
\]
\end{defn}
By \Cref{cor:poly}, with the choice of $c_0$ that cancels the term of order $r^{-N}(\log\frac{1}{\eps r})^{-N/2}$, one immediately obtains
\begin{lem}[Error of first approximation]
\label{lem:u1}
We have
\[
E_{1,\eps}
:=-\Delta u_{1}^{\eps}-(u_{1}^{\eps})^{\frac{N}{N-2}}
=-\dfrac{N(N-2)}{4}c_0
	\dfrac{1}{r^N(\log\frac{1}{\eps r})^{\frac{N+2}{2}}}.
\]
\end{lem}

Since the parameter $(N+2)/2=\nu+1$ in \Cref{lem:u1} is critical for the existence of a super-solution (see \eqref{eq:supsol-nu}; where one needs $\nu>N/2$), we will improve it by adding a log-polyhomogeneous correction.

\begin{defn}
The \emph{second approximation solution} is $u_{2}^{\eps}$ defined by
\[
u_{2}^{\eps}(r)
:=
	u_{1}^{\eps}(r)
	+c_1\phi_{N-2,\frac{N}{2},-1}^{\eps}(r)
=
	c_0\phi_{N-2,\frac{N-2}{2}}^{\eps}(r)
	+c_1\phi_{N-2,\frac{N}{2},-1}^{\eps}(r),
\]
with
\[
c_1=-\frac{N}{4}c_0.
\]
\end{defn}

\begin{lem}[Error of second approximation]
\label{lem:u2}
\[
E_{2,\eps}:=
-\Delta u_2^{\eps}-(u_2^\eps)^{\frac{N}{N-2}}
=O\left(
	\phi_{N,\frac{N+4}{2},-2}
\right).
\]
\end{lem}

\begin{proof}
By \Cref{lem:log-poly}, we have
\[\begin{split}
-\Delta u_2^{\eps}
&=
	c_0\left(
		\frac{(N-2)^2}{2}\phi_{N,\frac{N}{2}}^{\eps}
		-\frac{N(N-2)}{4}\phi_{N,\frac{N+2}{2}}^{\eps}
	\right)\\
&\quad\;
	-\frac{N}{4}c_0\left(
		\frac{N(N-2)}{2}\phi_{N,\frac{N+2}{2},-1}^{\eps}
		-(N-2)\phi_{N,\frac{N+2}{2}}^{\eps}
		+O\left(
			\phi_{N,\frac{N+4}{2},-1}^{\eps}
		\right)
	\right).
\end{split}\]
By binomial theorem and the choice of $c_1$,
\[\begin{split}
(u_2^{\eps})^{\frac{N}{N-2}}
&=
	\left(c_0\phi_{N-2,\frac{N-2}{2}}^{\eps}\right)^{\frac{N}{N-2}}
	\left(
		1-\dfrac{N}{4}
			\dfrac{
				\log\log\frac{1}{\eps r}
			}{
				\log\frac{1}{\eps r}
			}
	\right)^{\frac{N}{N-2}}\\
&=
	\dfrac{(N-2)^2}{2}c_0
	\phi_{N,\frac{N}{2}}^{\eps}
	\left(
		1-\dfrac{N^2}{4(N-2)}
			\dfrac{
				\log\log\frac{1}{\eps r}
			}{
				\log\frac{1}{\eps r}
			}
		+O\left(
			\dfrac{
				(\log\log\frac{1}{\eps r})^2
			}{
				(\log\frac{1}{\eps r})^2
			}
		\right)
	\right).
\end{split}\]
The proof is completed by taking the difference.
\end{proof}

Now we want to extend $u_2$ globally by $0$ outside the unit ball.
Let $\chi_{*}(r)$ be a smooth radial cut-off function supported on $B_{1}$ such that $\chi_{*}=1$ in $B_{1/2}$ and $|\nabla\chi_{*}|\leq C$.


\begin{defn}
The \emph{third approximation} $u_3^{\eps}$ is defined by
\[
u_3^{\eps}(r)=u_{2}^{\eps}(r) \chi_{*}(r),
    \quad \forall r>0.
\]
\end{defn}

\begin{prop}[Error of global approximate solution]
\label{prop:u3}
We have
\[
E_{3,\eps}
:=-\Delta u_{3}^{\eps}-(u_{3}^{\eps})^{\frac{N}{N-2}}
=
	O\left(
		\phi_{N,\frac{N+4}{2},-2}^{\eps}
	\right)
	\oneset{0<r\leq 1/2}
	+O\left(
		|\log\eps|^{-\frac{N-2}{2}} 
	\right)
	\oneset{1/2<r<1}.
\]
In particular,
\[
\norm[N,\frac{N+3}{2}]{E_{3,\eps}}
\leq C|\log\eps|^{-\frac12}.
\]
\end{prop}
Recall that the weighted spaces $L^\infty_{\mu,\nu}(B_{1/2})$ are defined in \eqref{eq:spaces}. Hereafter $\mathbf{1}_{A}$ denotes the characteristic function of a set $A$. In particular, the first and second terms of the error are supported respectively on the ball $B_{1/2}$ and on the annulus $B_{1}\setminus B_{1/2}$.

\begin{proof}
By \Cref{lem:u2},
\[\begin{split}
&\quad\;
	-\Delta u_{3}^{\eps}-(u_{3}^{\eps})^{\frac{N}{N-2}}\\
&=
	-\Delta u_{2}^{\eps} \chi_{*}
	-2\nabla u_{2}^{\eps} \cdot \nabla \chi_{*}
	-u_{2}^{\eps} \Delta \chi_{*}
	-(u_{2}^{\eps})^{\frac{N}{N-2}} \chi_{*}^{\frac{N}{N-2}}\\
&=
	\left(
		-\Delta u_2^\eps-(u_2^\eps)^{\frac{N}{N-2}}
	\right)\chi_{*}
	+(u_2^\eps)^{\frac{N}{N-2}}
	\left(
		\chi_{*}-\chi_{*}^{\frac{N}{N-2}}
	\right)
	+O\left(
		|u_2^\eps|+|(u_2^\eps)_r|
	\right)\oneset{1/2<r<1}\\
&=
	O\left(
		\phi_{N,\frac{N+4}{2},-2}^{\eps}\right)
    \oneset{0<r<1}\\
&\quad\;
	+O\left(
        \phi_{N,\frac{N}{2}}^{\eps}
        +\phi_{N,\frac{N+2}{2},-1}^{\eps}
        +\phi_{N,\frac{N+4}{2},-2}^{\eps}
        +\phi_{N-2,\frac{N-2}{2}}^{\eps}
        +\phi_{N-2,\frac{N}{2},-1}^{\eps}
        +\phi_{N-1,\frac{N+2}{2}}^{\eps}
    \right)
	\oneset{1/2<r<1}\\
&=
	O\left(
		\phi_{N,\frac{N+4}{2},-2}^{\eps}
        \oneset{0<r\leq 1/2}
		+\dfrac{1}{
            (\log\frac{1}{\eps})^{\frac{N-2}{2}}
        }
		\oneset{1/2<r<1}
	\right).
\end{split}\]
\end{proof}

\subsection{The linearized operator}

We look for a true solution in the form $u=u_{3}^{\eps}+\varphi$, where $\varphi$ is less singular than $u_{3}^{\eps}$ near the origin and bounded elsewhere. Hence $u$ behaves like $u_{3}^{\eps}$ and is singular exactly at the origin. Note that we do not impose $u>0$ away from the origin. Then the equation
\[
-\Delta u = |u|^{\frac{2}{N-2}}u
	\quad \textin B_{1}
\]
is equivalent to
\begin{equation}\label{L_eps}
L_{\eps}\varphi:=
	-\Delta\varphi
	-\frac{N}{N-2}(u_{3}^{\eps})^{\frac{2}{N-2}}\varphi
=-E_{3,\eps}+\cN[\varphi]
	\quad \textin B_{1}\subset \R^N,
\end{equation}
where $E_{3,\eps}$ is given in Lemma \ref{prop:u3} and
\begin{equation}\label{eq:cN}
\cN[\varphi]
=\abs{u_3^\eps+\varphi}^{\frac{2}{N-2}}
	(u_3^\eps+\varphi)
    -(u_3^\eps)^{\frac{N}{N-2}}
    -\frac{N}{N-2}(u_3^\eps)^{\frac{2}{N-2}}\varphi
\end{equation}
Note that $E_{3,\eps}=0$ on $\p B_{1}$ and $\varphi$ can be chosen such that $\varphi=0$ on $\p B_{1}$. 
\begin{remark}\label{rem:L}
Let us note that the linear operator $L_{\eps}$, defined in \eqref{L_eps}, can be written as
$$L_{\eps}\varphi=
	-\Delta\varphi
	-\dfrac{N(N-2)}{2}
	\phi_{2,1}^{\eps}\chi_{*}^{\frac{2}{N-2}}\varphi
	\left(
		1-\dfrac{N}{2(N-2)}
			\dfrac{
				\log\log\frac{1}{\eps r}
			}{
				\log\frac{1}{\eps r}
			}
		+O\left(
			\dfrac{
				(\log\log\frac{1}{\eps r})^2
			}{
				(\log\frac{1}{\eps r})^2
			}
		\right)
	\right),$$
as $\eps r\searrow 0$. Thus, we can assert that asymptotically as $\eps r\searrow 0$,
\[
L_{\eps}\varphi
=
	-\Delta\varphi
	-\frac{N(N-2)}{2}
    (1+o(1))
		\dfrac{
			\chi_{*}^{\frac{2}{N-2}}
		}{
			r^2\log\frac{1}{\eps r}
		}
	\varphi.
\]
\end{remark}
\begin{lem}[Super-solution]
\label{lem:local-supsol-radial}
There exists $\eps_1\in(0,e^{-1})$
such that for any $\nu\in[\frac{N}{2}+\frac14,\frac{N}{2}+1]$ and $\eps\in(0,\eps_1)$, $\phi_{N-2,\nu}^{\eps}$ is a super-solution for $L$. More precisely,
\[
L_{\eps}\phi_{N-2,\nu}^{\eps}
>
    \dfrac{N-2}{8}
    \phi_{N,\nu+1}^{\eps},
	\quad \textfor r\in(0,1).
\]
\end{lem}

\begin{proof}
By \Cref{cor:poly}, for any $r\in(0,1)$,
\begin{equation}\label{eq:supsol-nu}\begin{split}
L_{\eps} \phi_{N-2,\nu}^{\eps}
&=
    (N-2)\nu\phi_{N,\nu+1}^\eps
    -\nu(\nu+1)\phi_{N,\nu+2}^\eps
    -\dfrac{N(N-2)}{2}
    	\dfrac{
			\chi_{*}^{\frac{N}{N-2}}
		}{
			r^N(\log\frac{1}{\eps r})^{\nu+1}
		}\\
&\quad\;+
	\phi_{N,\nu+1}^{\eps}\chi_{*}^{\frac{2}{N-2}}
	\left(
		\dfrac{N^2}{4}
			\dfrac{
				\log\log\frac{1}{\eps r}
			}{
				\log\frac{1}{\eps r}
			}
		+O\left(
			\dfrac{
				(\log\log\frac{1}{\eps r})^2
			}{
				(\log\frac{1}{\eps r})^2
			}
		\right)
	\right)
		\\
&\geq
    \phi_{N,\nu+1}^\eps
	\left[
		(N-2)\left(\nu-\frac{N}{2}\right)
		-\dfrac{\nu(\nu+1)}{\log\frac{1}{\eps r}}
		-
			\dfrac{C
				(\log\log\frac{1}{\eps r})^2
			}{
				(\log\frac{1}{\eps r})^2
			}
		\right]\\
&>
    \dfrac{N-2}{8}\phi_{N,\nu+1}^\eps,
\end{split}\end{equation}
for all sufficiently small $\eps$.
\end{proof}

\begin{lem}[\emph{A priori} estimates]
\label{lem:apriori-new}
If $\varphi\in L^\infty_{N-2,\frac{N+1}{2}}(B_1)$ solves
\[\begin{cases}
L_\eps\varphi=f
    & \textin B_1\setminus\set{0},\\
\varphi=0
    & \texton \p B_1,
\end{cases}\]
with $\norm[N,\frac{N+3}{2}]{f}<\infty$, then
\[
\norm[N-2,\frac{N+1}{2}]{\varphi}
\leq 
    8\norm[N,\frac{N+3}{2}]{f}.
\]
\end{lem}

\begin{proof}
For any given $\delta\in(0,1)$, we can define the function
\[
\varphi^{\delta,\pm}
:=
    \frac{8}{N-2}
        \norm[N,\frac{N+3}{2}]{f}
        \phi_{N-2,\frac{N+1}{2}}^{\eps}
    +\delta\phi_{N-2,\frac{N}{2}+\frac14}^\eps
    \pm\varphi
\]
which satisfies
\[\begin{cases}
L\varphi^{\delta,\pm}
>
    \dfrac{N-2}{8}\delta
        \phi_{N,\frac{N}{2}+\frac34}^{\eps}
    & \textin B_1\setminus\set{0},\\
\varphi^{\delta,\pm}
\geq0
    & \texton \p B_1,\\
\varphi^{\delta,\pm}
>0
    & \text{ a.e. in } B_{r_1},
\end{cases}\]
where $r_1=r_1(\delta,\norm[N-2,\frac{N+1}{2}]{\varphi})>0$ is chosen small enough. Invoking \Cref{prop:L1MP}, $\varphi^{\delta,\pm}\geq0$  a.e. in $B_1$ and, by taking $\delta\searrow 0$,
\[
|\varphi|
\leq \frac{8}{N-2}
    \norm[N,\frac{N+3}{2}]{f}
    \phi_{N-2,\frac{N+1}{2}}^{\eps}.
\]
Since $N-2\geq1$, the proof is complete.
\end{proof}

\begin{lem}[The Poisson equation]\label{Poisson}
For any radial $f\in L^\infty_{N,\frac{N+3}{2}}(B_1)$, there exists a unique radial $\varphi\in L^\infty_{N-2,\frac{N+1}{2}}(B_1)$ solving
\[\begin{cases}
-\Delta \varphi=f
    & \textin B_1\setminus\set{0},\\
\varphi=0
    & \texton \p B_1.
\end{cases}\]
Moreover, there holds the estimate
\[
\norm[N-2,\frac{N+1}{2}]{\varphi}
\leq 
    8\norm[N,\frac{N+3}{2}]{f}.
\]
\end{lem}

\begin{proof}
Since the maximum principle implies uniqueness, we can assume $\varphi$ is radial. By direct integration\footnote{or the representation by Green formula},
\[
\varphi(r)
=
    \int_{r}^{1}
        t^{1-N}\int_{0}^{t}
            s^{N-1}f(s)
        \,ds
    \,dt,
\]
showing that $(-\Delta)^{-1}:L^\infty_{N,\frac{N+3}{2}}(B_1) \to L^\infty_{N-2,\frac{N+1}{2}}(B_1)$ is a well-defined bounded linear operator. The estimate follows from \Cref{lem:apriori-new} which also applies in the absence of the potential.
\end{proof}

\begin{prop}[Linear theory]
\label{prop:exist-new}
For any radial $f\in L^\infty_{N,\frac{N+3}{2}}(B_1)$, there exists a unique radial $\varphi\in L^\infty_{N-2,\frac{N+1}{2}}(B_1)$ solving
\[\begin{cases}
L_\eps\varphi=f
    & \textin B_1\setminus\set{0}\\
\varphi=0
    & \texton \p B_1.
\end{cases}\]
Moreover, there holds the estimate
\[
\norm[N-2,\frac{N+1}{2}]{\varphi}
\leq 
    8\norm[N,\frac{N+3}{2}]{f}.
\]
In other words, $L_\eps^{-1}:L^\infty_{N,\frac{N+3}{2}}(B_1) \to L^\infty_{N-2,\frac{N+1}{2}}(B_1)$ is a bounded linear operator with a uniformly bounded operator norm,\footnote{By definition, $\| L_{\eps}^{-1} \| = \sup \{\| L_{\eps}^{-1} u \|_{L^\infty_{N-2,\frac{N+1}{2}}(B_1)}:\ u\in L^\infty_{N,\frac{N+3}{2}}(B_1),  \| u \|_{L^\infty_{N,\frac{N+3}{2}}}(B_1)=1\}$.}

\[
\norm[]{L_\eps^{-1}}\leq 8.
\]
\end{prop}

\begin{proof}
We can prove it using the method of continuity (see for example \cite[Theorem 5.2]{GT}). Indeed, if we interpolate between $-\Delta$ and $L_\eps$ linearly, i.e. for any $\lambda\in[0,1]$, we define
\[
L_{\eps}^\lambda
:=-\Delta+\lambda\frac{N}{N-2}(u_3^\eps)^{\frac{2}{N-2}},
\]
we just need to show that $L_\eps^\lambda$ has a bounded inverse for all $\lambda\in[0,1]$ from $L^\infty_{N,\frac{N+3}{2}}(B_1)$ to $L^\infty_{N-2,\frac{N+1}{2}}(B_1)$. We proceed by induction, increasing $\lambda$ by a fixed amount and iterate. By \Cref{Poisson}, the assertion is true when $\lambda=0$. If $(L^\lambda_{\eps})^{-1}:L^\infty_{N,\frac{N+3}{2}}(B_1) \to L^\infty_{N-2,\frac{N+1}{2}}(B_1)$ exists, then for any $\delta\in(0,1-\lambda]$, the equation
\[
L_{\eps}^{\lambda+\delta}\varphi
=L_{\eps}^\lambda\varphi
    +\delta\frac{N}{N-2}(u_\eps)^{\frac{2}{N-2}}\varphi
=f
\]
can be rewritten (in its fixed-point form) as
\begin{equation}\label{eq:cont}
\varphi=(L_{\eps}^\lambda)^{-1}f-\delta\frac{N}{N-2}
    (L_{\eps}^\lambda)^{-1}\left(
        (u_\eps)^{\frac{2}{N-2}}\varphi
    \right).
\end{equation}
Note that the multiplication operator by $(u_3^\eps)^{\frac{2}{N-2}}$ maps $L^\infty_{N,\frac{N+3}{2}}(B_1) \to L^\infty_{N-2,\frac{N+1}{2}}(B_1)$ and is bounded. Then, in view of \Cref{lem:apriori-new}, for $\delta$ universally small the right hand side of \eqref{eq:cont} defines a contraction, showing that $(L_{\eps}^{\lambda+\delta})^{-1}$ exists (which again has the same bound by \Cref{lem:apriori-new}). The invertibility of $L_\eps$ follows after $\delta^{-1}$ iterations.
\end{proof}

%

\subsection{The nonlinear equation}
We are in a position to solve the equation
\[\begin{cases}
L_\eps\varphi=-E_{3,\eps}+\cN[\varphi]
    & \textin B_{r_*},\\
\varphi=0
    & \texton \p B_{r_*},
\end{cases}\]
where $E_{3,\eps}$ is the error of $u_3^\eps$ given in \Cref{prop:u3} and the superlinear term $N$ is defined in \eqref{eq:cN}. The non-linear equation, in the fixed point form, reads
\begin{equation*}
\varphi=G_\eps[\varphi]
:=L_\eps^{-1}(-E_{3,\eps}+\cN[\varphi]),
\end{equation*}
where the solution operator $L_\eps^{-1}:L^\infty_{N,\frac{N+3}{2}}(B_1) \to L^\infty_{N-2,\frac{N+1}{2}}(B_1)$ is defined in Lemma \ref{prop:exist-new}.
We consider the Banach space
\[
X\equiv
X_{\bar{C},\eps}:=\set{
    \varphi\in L^1(B_{r_*})|
    \norm[X]{\varphi}
    :=\norm[N-2,\frac{N+1}{2}]{\varphi}
    \leq
        \bar{C}|\log\eps|^{-\frac12}
},
\]
where $\bar{C}$ is a positive constant that will be characterized in the following proposition.
\begin{prop}[Contraction]
\label{prop:G-new}
There exists $\bar{C}>0$ and $\bar{\eps}\in(0,\eps_2)$ such that for any $\eps\in(0,\bar{\eps})$, $G_\eps:X_{\bar{C}}\to X_{\bar{C}}$ and $G_\eps$ is a contraction.
\end{prop}

\begin{proof}
By \Cref{prop:u3} and \Cref{prop:exist-new},
{
\[
\norm[N,\frac{N+3}{2}]{E_{3,\eps}}
\leq C|\log\eps|^{-\frac12}, \quad
\norm[X]{L_\eps^{-1}( E_{3,\eps})}\leq C_1|\log\eps|^{-\frac12}.
\]}
From \eqref{eq:cN}, for any $\varphi,\tilde\varphi\in X$, the Fundamental Theorem of Calculus assures that

\[\begin{split}
\cN[\varphi]-\cN[\tilde\varphi]
=&\abs{u_3^\eps+\varphi}^{\frac{2}{N-2}}(u_3^\eps+\varphi)-\abs{u_3^\eps+\tilde{\varphi}}^{\frac{2}{N-2}}(u_3^\eps+\tilde{\varphi})
    -\frac{N}{N-2}(u_3^\eps)^{\frac{2}{N-2}}(\varphi-\tilde{\varphi})
 \\ =&  \dfrac{N}{N-2}
	\int_0^1\left(|u_3^\eps+(1-t)\varphi-t\tilde\varphi|^{\frac{2}{N-2}}
		-(u_3^\eps)^{\frac{2}{N-2}}
	\right)\,dt
	\cdot(\varphi-\tilde\varphi).
\end{split}\]

Since the function $|\cdot|^{\frac{2}{N-2}}$ is uniformly $C^{0,\frac{2}{N-2}}$ for $N\geq 4$,
\[\begin{split}
\abs{\cN[\varphi]-\cN[\tilde\varphi]}
&\leq
	C\left(
		\abs{\varphi}^{\frac{2}{N-2}}
		+\abs{\tilde\varphi}^{\frac{2}{N-2}}
	\right)
	|\varphi-\tilde\varphi|\\
&\leq
	C(\bar{C}|\log\eps|^{-\frac12})^{\frac{2}{N-2}}
	(\phi_{N-2,\frac{N+1}{2}}^\eps)^{\frac{N}{N-2}}
	\norm[X]{\varphi-\tilde\varphi}\\
&\leq
	C\bar{C}^{\frac{2}{N-2}}
    |\log\eps|^{-\frac{4}{N-2}}
	\phi_{N,\frac{N+3}{2}}^{\eps}
	\norm[X]{\varphi-\tilde\varphi},
\end{split}\]
where last inequality follows from the elementary fact that
\[
\frac{N+1}{2}\frac{N}{N-2}-\frac{N+3}{2}
=\frac{3}{N-2}.
\]
By \Cref{prop:exist-new},
\[
\norm[X]{{L_\eps^{-1}(\cN[\varphi]-\cN[\tilde\varphi])}}
\leq
    C\bar{C}^{\frac{2}{N-2}}
    |\log\eps|^{-\frac{4}{N-2}}
    \norm[X]{\varphi-\tilde\varphi}.
\]
When $N=3$,
\[\begin{split}
\abs{\cN[\varphi]-\cN[\tilde\varphi]}
&\leq
	C\left(
		\abs{\varphi}
		+\abs{\tilde\varphi}
	\right)
    |u_3^\eps|
	|\varphi-\tilde\varphi|\\
&\leq
	C(\bar{C}|\log\eps|^{-\frac12})
    (\phi_{1,2}^\eps)^2
	\phi_{1,\frac12}^{\eps}
	\norm[X]{\varphi-\tilde\varphi}\\
&\leq
	C\bar{C}
    |\log\eps|^{-2}
	\phi_{3,3}^{\eps}
	\norm[X]{\varphi-\tilde\varphi},
\end{split}\]
so that
\[
\norm[X]{
L_\eps^{-1}(\cN[\varphi]-\cN[\tilde\varphi])}
\leq
    C\bar{C}
    |\log\eps|^{-2}
    \norm[X]{\varphi-\tilde\varphi}.
\]
Hence, by first choosing $\bar{C}=2C_1$ and then $\eps$ small, we know that $G_\eps:X\to X$ (by specializing $\tilde\varphi=0$) and $G_\eps$ is a contraction.
\end{proof}

\subsection{Proof of \Cref{thm:2}}

By \Cref{prop:G-new}, there exists a singular solution of
\[
-\Delta u=|u|^{\frac{2}{N-2}}u
	\quad \textin B_{r_*},
\]
possibly sign-changing, that behaves like $u_3^\eps$ (in particular positive) near the origin. By the scaling invariance, the desired solution is then
\[
\bar{u}(x)=\eps^{N-2}u(\eps x),
\]
which is defined in $B_{r_*/\eps}\setminus\set{0}$ and positive in $B_1\setminus\set{0}$, for all small enough $\eps$.
\qed

\medskip

\section{Singular Yamabe problem}
\label{sec:Yamabe}
This last Section is dedicated to the construction of a solution, which is singular along a submanifold $\Sigma$ of dimension $\frac{n-2}{2}$, for the Yamabe problem. As we mentioned before, this dimension is maximal for the singular set (see the classical work of Schoen and Yau \cite{SY} for details), so it can be considered as a critical case. The existence of complete metrics solving the problem is already known by Mazzeo and Pacard  in \cite{Pacard-1,MP}, but it is interesting observe how the previous study can be also used to construct, in a simpler way, solutions for this critical case. The main difference with the cited works is the abscence of weighted H\"{o}lder spaces. We carry out all the estimates in weighted $L^{\infty}$ spaces. Moreover, we emphasise again, that this procedure will let us construct such solutions also for the non-local case (see the forthcoming paper \cite{CD2}).

Given $\Sigma\subset \R^n$, the Yamabe singular Yamabe problem is equivalent to find a positive solution to
\[
-\Delta u = u^{\frac{n+2}{n-2}}
	\quad \textin \R^n\setminus \Sigma.
\]
If we consider the model case $\Sigma=\R^k$, our problem can be rewritten, with $N:=n-k=\frac{n+2}{2}$ as
\[
-\Delta u = u^{\frac{N}{N-2}}
	\quad \textin \R^N\setminus \{0\}.
\]
Inspired by previous works in the local and non local case (see \cite {fat,MP,Pacard-1}), we will use the solution of this model case, as an approximate solution for a general submanifold $\Sigma$.

As in \Cref{Sec:sing}, we  will denote by $\cT_{r}$ the tubular neighbourhood of width $r>0$ around $\Sigma$ and, here, we will restrict to the construction of the solution on  $\textin \cT_{r_*}\setminus\Sigma$, that we will identify with $\textin (0,r_*)\times \bS^{N-1}\times \Sigma$. Note that to have a solution in the whole $\R^n\setminus \Sigma$ it is enough to glue the resulting metric to the Euclidean one far away from the singularity.

Using the Fermi coordinates, we consider an \emph{Ansatz} depending only the normal variable and repeat the procedure in \Cref{sec:radial}. \Cref{thm:2} gives an exact solution on a ball centered at the singularity, so that there will be no error when the cut-off introduced in \eqref{eq:yamabe-cutoff} 
equals $1$. However, the curvature of the singular submanifold $\Sigma$ will enter here.

\subsection{\emph{Ansatz} and strategy}

Let $\bar{u}_\eps(r)$ be the solution given by \Cref{thm:2}, with $\eps\in(0,\bar\eps]$ small. We set
\begin{equation}\label{eq:yamabe-cutoff}
v_{\eps}(r,\omega,y)
=v_\eps(r)
=\bar{u}_\eps(r)\chi_{*}(r),
\end{equation}
which is supported on $\cT_{r_*}$. By \eqref{eq:u-bar-asymp}, it is easy to see that
\begin{equation}\label{eq:v-eps-asymp}
v_\eps(r)\asymp
\dfrac{\chi_{*}(r)}{
    r^{N-2}
    (\log\frac{1}{\eps r})
        ^{\frac{N-2}{2}}
}
    \quad \textfor r\in(0,r_*),
\end{equation}
i.e. $v_\eps$ is bounded between positive multiples of the right hand side.
We consider a perturbation $\bar{v}:=v_\eps+\psi$, which will be a solution of
\begin{equation}\label{eq:v-bar}\begin{cases}
-\Delta_g \bar{v}=|\bar{v}|^{\frac{N}{N-2}}
    &  \textin (0,r_*)\times \bS^{N-1}\times \Sigma,\\
\bar{v}=0
    &  \texton \set{r_*}\times \bS^{N-1}\times \Sigma.
\end{cases}\end{equation}
when $\psi$ solves the linearized equation
\begin{equation}\label{eq:lin}
\cL\psi
=
	-\cE+\bar\cN[\psi].
\end{equation}
where
\begin{equation}\label{eq:cL-def}
\cL\psi
=
	-\Delta_g \psi
	-\frac{N}{N-2}v_\eps^{\frac{2}{N-2}}\psi
\end{equation}
\[
\cE
=
	-\Delta_g v_{\eps}-v_\eps^{\frac{N}{N-2}}
\]
\[
\bar\cN[\psi]
=
	|v_\eps+\psi|^{\frac{N}{N-2}}
	-v_\eps^{\frac{N}{N-2}}
	-\frac{N}{N-2}v_\eps^{\frac{2}{N-2}}\psi.
\]
It is important to remind here the following fact that will be repeatedly used along the Section:
\begin{remark}\label{rem:change-var}
The pair $(\psi,f)$ solves the Poisson equation
\[
\begin{cases}
-\Delta_g \psi	=f
	& \textin (0,r_*)\times \bS^{N-1}\times \Sigma,\\
\psi=0
	& \texton \set{r_*} \times \bS^{N-1}\times \Sigma.\\
\end{cases}
\]
if and only if $(\tilde\psi,\tilde f)=(\psi\circ\Phi^{-1},f\circ\Phi^{-1})$ solves
\[\begin{cases}
-\Delta \tilde\psi = \tilde f
	& \textin \cT_{r_*}\setminus\Sigma,\\
\tilde\psi=0
	& \texton \p\cT_{r_*},
\end{cases}\]
where $\Phi$ is the diffeomorphism given in \eqref{eq:Fermi-3} and $\Delta$ is the usual Laplacian with the flat metric.
\end{remark}

Then we know that a maximum principle holds for \eqref{eq:v-bar}, so $\bar{v}>0$, and we can also conclude that a positive solution of
\[\begin{cases}
-\Delta v=v^{\frac{N}{N-2}}
    & \textin \cT_{r_*},\\
v=0
    & \texton \p\cT_{r_*},
\end{cases}\]
is given by $v(z)=\bar{v}(\Phi^{-1}(z))$.

Our goal then, is to find the proper perturbation which solves \eqref{eq:lin}. We do it by a fixed point argument, as in the previous Section. However, here, we need to distinguish two cases depending on the dimension. First, for $N\geq 4$, we show that the error $\cE$ is bounded in the space $L^\infty_{N-1}$, that $\cL^{-1}:L^\infty_{N-1}\to L^\infty_{N-3}$ exists and it is a bounded linear operator, and that $\cG:L^\infty_{N-3}\to L^\infty_{N-3}$ defined by $\cG[\varphi]=\cL^{-1}(-\cE+\bar\cN[\psi])$ is a contraction.

Later, when $N=3$, we see a low dimension phenomenon, so we need to use the barrier $(\log\frac1r)^{\nu}$, $\nu\in(0,1)$.

We conclude the idea of the strategy with a remark on the choice of parameters. Depending on the geometry of $\Sigma$, we pick $r_*\in(0,1)$ such that the constants in \eqref{eq:Lap-g-rad} multiplied to $r_*$ are small, with respect to a dimensional constant (see \Cref{lem:global-supsol}). Once $r_*$ is fixed, the smallness will be controlled just using $\eps$.

\subsection{Error estimates}
In this Section we will show some explicit computation to prove the bound of the error
 \[ \cE
=
	-\Delta_g v_{\eps}-v_\eps^{\frac{N}{N-2}}
\]
made by approximating with $v_{\eps}$ as above.

\begin{lem}[Error of approximation]
\label{lem:cE}
For any $r\in(0,r_*)$, we have
\[
|\cE|
\leq
    \dfrac{C}{
        r^{N-1}
        (\log\frac{1}{\eps r})^{\frac{N-2}{2}}
    }.
\]
In particular,
\[
\norm[N-1]{\cE}\leq C|\log\eps|^{-\frac{N-2}{2}}
\]
for $N\geq 4$, and
\[
\norm[2,\frac14]{\cE}\leq C|\log\eps|^{-\frac14}
\]
when $N=3$.
\end{lem}

\begin{proof}
By \eqref{eq:Lap-g-rad},
\[\begin{split}
-\Delta_g v_\eps- v_\eps^{\frac{N}{N-2}}
&=
    -\Delta_r(\bar{u}_\eps\chi_{*})
    -(\bar{u}_\eps \chi_{*})^{\frac{N}{N-2}}
    +O(r)(\bar{u}_\eps\chi_{*})_{rr}
    +O(1)(\bar{u}_\eps\chi_{*})_{r}
    \\
&=
    \bar{u}_\eps^{\frac{N}{N-2}}
        (\chi_{*}-\chi_{*}^{\frac{N}{N-2}})
    -[(2+O(r))(\bar{u}_\eps)_r+O(1)\bar{u}_\eps]
        (\chi_{*})_r
    -\bar{u}_\eps \Delta_r \chi_{*}\\
&\quad\;
    +O(1)[r(\bar{u}_\eps)_{rr}+(\bar{u}_\eps)_r]\\
&=
    \dfrac{O(1)}{
        r^{N-1}(\log\frac{1}{\eps r})^{\frac{N-2}{2}}
    }\oneset{r<r_*/2}
    +O(1)|\log\eps|^{-\frac{N-2}{2}}
        \oneset{r_*/2<r<r_*}.
\end{split}\]
\end{proof}

Now we are in good shape to do the linear study but, as we mention before, we need to distinguish if $N=3$ or higher. Let us focus first in the case $N\geq 4$.
\subsection{Linear theory for $N\geq 4$}

Consider
\begin{equation}\label{eq:lin-theory}
\begin{cases}
\cL \psi=f
    & \textin (0,r_*)\times \bS^{N-1}\times \Sigma,\\
\psi=0
    & \texton \set{r_*}\times \bS^{N-1}\times \Sigma.
\end{cases}
\end{equation}

\begin{lem}[Global super-solution]
\label{lem:global-supsol}
For any fixed $\mu\in(0,N-2)$, there exists a small $r_*=r_*(\mu)\in(0,1)$ such that for any $r\in(0,r_*)$, we have
\begin{equation}\label{eq:supsol-mu}
\cL r^{-\mu}
\geq
    \dfrac{\mu(N-2-\mu)}{2}r^{-\mu-2}.
\end{equation}
\end{lem}

\begin{proof}
By \eqref{eq:cL-def}, \eqref{eq:Lap-g-rad} and \eqref{eq:v-eps-asymp}, we compute
\begin{equation*}
\begin{split}
\cL r^{-\mu}
&\geq
    -\Delta_r r^{-\mu}
    +O(r)r^{-\mu-2}+O(1)r^{-\mu-1}
    -\dfrac{C}{r^2\log\frac{1}{\eps r}}r^{-\mu}\\
&\geq
    \left(
        \mu(N-2-\mu)
        -Cr_*
        -\dfrac{C}{
            \log\frac{1}{\bar\eps r_*}
        }
    \right)r^{-(\mu+2)},
\end{split}\end{equation*}
for a universal constant $C$. Therefore $r_*$ can be chosen small enough such that we have a super-solution.
\end{proof}

\begin{remark}\label{rem:r*}
Hereafter we fix $r_*$ such that \eqref{eq:supsol-mu} holds true for $\mu=N-3$ and $\mu=N-5/2$.
\end{remark}

\begin{lem}[\emph{A priori} estimates]
\label{lem:apriori}
If $\psi\in L^{\infty}_{N-3}\left([0,r_*)\times \bS^{N-1}\times \Sigma\right)$ is a solution of \eqref{eq:lin-theory} with $\norm[N-1]{f}<\infty$, then
\[
\norm[N-3]{\psi}
\leq 2\norm[N-1]{f}.
\]
\end{lem}

\begin{proof}
For any $\delta>0$ we can define the functions
\begin{equation}\label{eq:apriori}
\psi^{\delta,\pm}
=
	\frac{2}{N-3}\norm[N-1]{f} r^{-(N-3)}
	+\delta r^{-(N-\frac52)}
	\pm \psi,
\end{equation}
which, by Lemma \ref{lem:global-supsol}, satisfy
\[\begin{cases}
\cL \psi^{\delta,\pm}
\geq 
	\frac{2N-5}{8}\delta r^{-(N-\frac12)}
>0
	& \textin (0,r_*)\times \bS^{N-1}\times \Sigma,\\
\psi^{\delta,\pm}
=\frac{2}{N-3}\norm[N-1]{f} r_*^{-(N-3)}
	+\delta r_*^{-(N-\frac52)}
>0
	& \texton \set{r_*}\times \bS^{N-1}\times \Sigma,\\
\psi^{\delta,\pm}
\geq r^{-(N-\frac52)}\left(
		\delta-\norm[N-3]{\psi}\sqrt{r}
	\right)
>0
	& \textin (0,r_0)\times \bS^{N-1}\times \Sigma,\\
\end{cases}\]
for some $r_0=r_0(\delta,\norm[N-3]{\psi})>0$. We apply now the maximum principle given in  \Cref{prop:L1MP} to get $\psi^{\delta,\pm}\geq0$. By taking $\delta\searrow0$, $\psi^{0,\pm}\geq 0$, as desired.
\end{proof}

Concerning the existence of solutions of \eqref{eq:lin-theory}, we use the method of continuity \cite[Theorem 5.2]{GT} and consider the linearly interpolated operators
\[
\cL_\lambda
=
	-\Delta_g
	-\lambda\frac{N}{N-2}v_\eps^{\frac{2}{N-2}},
\]
for $\lambda\in[0,1]$, and the family of equations
\begin{equation}\label{eq:L-lambda}
\begin{cases}
\cL_\lambda \psi
	=-\Delta_g \psi-\lambda\frac{N}{N-2}v_\eps^{\frac{2}{N-2}}\psi
	=f
	& \textin (0,r_*)\times \bS^{N-1}\times \Sigma,\\
\psi=0
	& \texton \set{r_*} \times \bS^{N-1}\times \Sigma.\\
\end{cases}
\end{equation}
It is clear that \Cref{lem:apriori} also holds when $\cL=\cL_1$ is replaced by $\cL_\lambda$, with a constant uniform in $\lambda\in|0,1]$. The reason why it is enough to consider weighted $L^\infty$ spaces only, lies in the fact that $\cL_1-\cL_0$ is a zeroth order term, where no extra regularity is necessarily to make sense of the functions involved. Therefore, it suffices to solve \eqref{eq:L-lambda} when $\lambda=0$ in order to start the iteration. 

\begin{lem}[Existence for $\lambda=0$]
\label{lem:lap-inv}
The operator $\cL_0=(-\Delta_g)$ is invertible and
\[
(-\Delta_g)^{-1}:L^\infty_{N-1}((0,r_*)\times \bS^{N-1}\times \Sigma)\to L^\infty_{N-3}((0,r_*)\times \bS^{N-1}\times \Sigma)
\]
is a bounded linear operator, i.e., there exists a constant $C_*=C_*(r_*)$ such that $\|(-\Delta_g)^{-1}\|\leq C_*$.
\end{lem}

\begin{proof}
By \Cref{rem:change-var}, we can work with the flat metric considering the problem in $\cT_{r_*}$ . Thus, let $\tilde f\in L^1(\cT_{r_*})$ with
\[
\sup_{z\in\cT_{r_*}}\dist(z,\Sigma)^{N-1}|\tilde f(z)|<\infty.
\]
 We need to show that
\[
\sup_{z\in\cT_{r_*}}\dist(z,\Sigma)^{N-3}|G_{\cT_{r_*}}\ast\tilde f(z)|<\infty,
\]
where $G_{\cT_{r_*}}$ is the Green function associated to $-\Delta$ in $\cT_{r_*}$.

First, we observe that in $\cT_{r_*}\setminus\cT_{r_*/2}$ the weight does not play any role in the finiteness and, it is standard that $G_{\cT_{r_*}}\ast \tilde f$ is bounded there. Then, we only need to prove the bound in $\cT_{r_*/2}$, where for any $z, \bar{z}\in \cT_{r_*}$, $G_{\cT_{r_*}}(z,\bar{z})$ is comparable to $|z-\bar z|^{-(N-2)}$. Using now polar coordinates and the diffeomorphism $\Phi$ defined in \eqref{eq:Fermi-3}, we can rewrite it in Fermi coordinates by $z=\Phi(r,\omega,y)$ and $\bar z=\Phi(\bar r,\bar\omega,\bar y)$, and so it suffices to show the finiteness of
\[
I=r^{N-3}\int_{\bar{z}\in\cT_{r_*/2}}
		\dfrac{1}{|z-\bar{z}|^{N-2}}
		\dfrac{1}{\bar{r}^{N-1}}
	\,d\bar{z},\quad\forall z\in \cT_{r_*}.
\]

We have now a singular integral expression, but we observe that the kernel is regular unless $z$ and $\bar{z}$ are close. Using polar coordinates as in \eqref{eq:Fermi-2},  we can write $z=(x,y)$ and $\bar{z}=(\bar x,\bar y)$, where $x=r\omega$, $\bar{x}=\bar{r}\bar\omega$. So we can take $\delta_*>0$ as small as desired and we have
\[
I\leq C(\delta_*)+Cr^{N-3}
	\int_{|x-\bar{x}|<\delta_*,|y-\bar{y}|<\delta_*}
		\dfrac{1}{(|x-\bar{x}|^2+|y-\bar{y}|^2)^{\frac{N-2}{2}}}
		\dfrac{1}{\bar{r}^{N-1}}
	\,d\bar{x}\,d\bar{y},
\]
 By parametrizing $y\in\Sigma$ as a graph and integrating all over $\R^k$,
\[
I\leq C(\delta_*)
	+Cr^{N-3}\int_{|x-\bar{x}|<\delta_*}
		\dfrac{1}{|x-\bar{x}|^{N-2}}
		\dfrac{1}{\bar{r}^{N-1}}
	\,d\bar{x}.
\]
Using polar coordinates and naming $\rho:=\frac{\bar{r}}{r}$, $\theta=\angle (\omega,\bar{\omega})$, the rotational invariance of the integrand asserts
\[
I\leq C(\delta_*)
	+\int_{0}^{\frac{r_*}{2r}}
		\int_{0}^{\pi}
			\dfrac{\sin^{N-2}\theta
			}{
				(1+\rho^2-2\rho\cos\theta)^{\frac{N-2}{2}}
			}
		\,d\theta
	\,d\rho
\leq C_*.
\]
This completes the proof.
\end{proof}

\begin{cor}[Existence]
\label{cor:exist-N>3}
For any $f\in L^\infty_{N-1}$, there exists a unique solution $\psi\in L^\infty_{N-3}$ of \eqref{eq:lin-theory}, satisfying \eqref{eq:apriori}. In other words, $\cL^{-1}:L^\infty_{N-1}\to L^\infty_{N-3}$ is a bounded linear operator with $\|\cL^{-1}\|\leq C_*$, where $C_*$ is the constant given by \Cref{lem:lap-inv}.
\end{cor}

\begin{proof}
If we choose $\lambda\in[0,1)$ such that $\cL_\lambda^{-1}$ is invertible, then the equation
\[
\cL_{\lambda+\delta} u=f
\]
is equivalent to
\[
u=\cL_\lambda^{-1}f
	-\delta\frac{N}{N-2}
	\cL_\lambda^{-1}\left(v_\eps^{\frac{2}{N-2}}u\right),
\]
which defines a contraction on $L^\infty_{N-3}$ if $\delta>0$ is small enough. Starting from $\lambda=0$ (\Cref{lem:lap-inv}) and using the \emph{a priori} estimates in \Cref{lem:apriori} we see that $\cL_1$ is invertible after $\delta^{-1}$ iterations.
\end{proof}


\subsection{The nonlinear equation for $N\geq 4$}

Knowing the invertibility of $\cL$ it is easy to solve \eqref{eq:lin} in the proper space. We write \eqref{eq:lin}  in the fixed point form
\[
\psi=\cG[\psi]:=\cL^{-1}(-\cE+\bar\cN[\psi]),
\]
where $\cE$ and $\cN[\psi]$ are given in \eqref{eq:cL-def}, and  we define the space
\[
\cX\equiv\cX_{C_2,\eps}
:=\set{
	v\in L^\infty_{N-3}\left(
		(0,r_*)\times \bS^{N-1}\times \Sigma
	\right):
	\norm[N-3]{v}\leq C_2|\log\eps|^{-\frac{N-2}{2}}
},
\]
where $C_2>0$ will be characterized in \Cref{prop:cG}.

Now are ready to prove the following:
\begin{prop}
\label{prop:cG}
There exists $C_2>0$ and $\eps_3\in(0,\bar{\eps})$ such that if $\eps\in(0,\eps_3)$, then
$\cG: \cX \to \cX$ and $\cG$ is a contraction in $\cX$.
\end{prop}

\begin{proof}
Let $\psi\in \cX$, we can estimate $\norm[\cX]{\cG[\psi]}$ as follows. First, from \Cref{lem:cE} and \Cref{cor:exist-N>3}
\[
\norm[\cX]{\cL^{-1}(\cE)}
\leq C\norm[]{\cL^{-1}}|\log\eps|^{-\frac{N-2}{2}}\leq
C C_*|\log\eps|^{-\frac{N-2}{2}}=:C_{**}|\log\eps|^{-\frac{N-2}{2}}.
\]
Now, proceeding as in the proof of \Cref{prop:G-new}, for $\psi,\bar\psi\in\cX$, we have
\[\begin{split}
\abs{\bar\cN[\psi]-\bar\cN[\bar\psi]}
&\leq C\left(
		|\psi|^{\frac{2}{N-2}}
		+|\bar\psi|^{\frac{2}{N-2}}
	\right)|\psi-\bar\psi|\\
&\leq C
	(C_2|\log\eps|^{-\frac{N-2}{2}})^{\frac{2}{N-2}}
	(r^{-(N-3)})^{\frac{N}{N-2}}
	|\psi-\bar\psi|
\end{split}\]
Since $(N-3)\frac{N}{N-2}>N-1$, by \Cref{cor:exist-N>3},
\[
\norm[\cX]{\cG[\psi]-\cG[\bar\psi]}
\leq CC_2^{\frac{2}{N-2}}\norm[]{\cL^{-1}}
	|\log\eps|^{-1}
	\norm[\cX]{\psi-\bar\psi},
\]
hence the result follows as in \Cref{prop:G-new} by taking $C_2=2C_{**}$ and $\eps$ small enough.
\end{proof}

\subsection{Proof of \Cref{thm:1} for $N\geq 4$}
By \Cref{prop:cG}, there exists a unique solution of \eqref{eq:v-bar} which satisfies
\[
\bar{v}(r,\omega,y)=\bar{u}(r)+O(r^{-(N-1)})
	\quad \textas r\searrow 0.
\]
By \Cref{rem:change-var}, $\tilde{v}(z)=\bar{v}(\Phi^{-1}(z))$ solves \eqref{eq:main1} and behaves like $\bar{u}(\dist(z,\Sigma))$ near $\Sigma$. In particular, it is positive near $\Sigma$ and bounded elsewhere in $\cT_{r_*}$. Since $\tilde{v}$ is super-harmonic, it cannot attain a local minimum in $\cT_{r_*}$. We conclude that $\tilde{v}>0$ in $\cT_{r_*}$ and is singular exactly on $\Sigma$, as desired.
\qed

\subsection{The case $N=3$}

As the scheme remains the same, we only indicate the modifications, due to the need of a logarithmic correction. Recall that the error $\cE$ is small in $\norm[2,\frac14]{\cdot}$. 
Then the super-solution in \Cref{lem:global-supsol} is replaced by $(\log\frac{1}{\eps r})^{\frac34}$, and the final integral in the proof of \Cref{lem:lap-inv} grows like $\log\frac{1}{r}$, showing that instead $\cL^{-1}:L^\infty_{2,\frac14} \to L^\infty_{0,-\frac34}$ is bounded. Thus, the fixed point argument implies the existence of a perturbation small in $L^\infty_{0,-\frac34}$.

\appendix

\section{$L^1$ theory}\label{Ap1}

\subsection{Maximum principle for positive operators}

Inspired by the classical $L^1$ theory due to Brezis, Cazenave, Martel and Ramiandrisoa \cite{BCMR}, and Dupaigne and Nedev \cite{Dupaigne-Nedev}, we prove a version of maximum principle in annular domains.

Let $0\in\Omega\subset\R^N$. Let $V:\Omega\setminus\set{0}\to\R$ be a (possibly) singular potential satisfying
\[
0\leq V(x)<\dfrac{(N-2)^2}{4}\dfrac{1}{|x|^2},
    \quad \forall x\in\Omega\setminus\set{0}.\footnote{The equality, i.e. the critical Hardy potential, can be allowed, see \cite{Dupaigne-Nedev}. The strict inequality suffices for our purpose, and the presentation is simpler.}
\]
Consider an operator $P$ of the form
\[
P=-\Delta-V(x),
\]
which is positive in the sense of having a positive first Dirichlet eigenvalue,
\[
\int_{\Omega}uPu\,dx
=\int_{\Omega}|\nabla u|^2\,dx
    -\int_{\Omega}V(x)u^2\,dx
\geq
    \lambda_1\int_{\Omega}u^2\,dx
        \quad \forall u\in C_c^\infty(\Omega),
\]
for some $\lambda_1>0$, via the Hardy--Poincar\'{e} inequality. Consider a very weak solution $u\in L^1(\Omega)$ 
for the Dirichlet problem
\begin{equation}\label{eq:Pu=f}\begin{cases}
Pu=f
    & \textin\Omega\\
u=g
    & \texton\p\Omega,
\end{cases}\end{equation}
with $f\in L^1(\Omega;\dist(x,\p\Omega)\,dx)$, $g\in C(\p\Omega)$ in the sense
\begin{equation}\label{eq:L1-sense}
\int_{\Omega}uP\zeta\,dx
=\int_{\Omega}f\zeta\,dx
    +\int_{\p\Omega}g\pnu{\zeta}\,d\sigma,
    \quad \forall \zeta\in C^2(\overline{\Omega})
    \text{ with } \zeta|_{\p\Omega}=0.
\end{equation}
Using the techniques of \cite[Lemma 1]{BCMR} and \cite[Lemma 1.1]{Dupaigne-Nedev}, we prove the following

\begin{prop}[Maximum principle]
\label{prop:L1MP}
If $f,g\geq 0$ and $u\geq0$ a.e. in some $B_\delta$ with $\delta>0$, then
\[
u\geq 0
    \quad \text{ a.e. in }\Omega.
\]
\end{prop}
First we need an existence result for functions with higher integrability.
\begin{lem}[Variational existence]
\label{lem:H1-exist}
For any datum $f\in H^{-1}(\Omega)$,
there exists a unique solution $u\in H^1_0(\Omega)$ to \eqref{eq:Pu=f}.
\end{lem}

\begin{proof}
This is a standard application of the Riesz Representation Theorem (see e.g. Theorem $5.7$ in \cite{GT}) on the bounded linear functional $f(v)=\int_{\Omega}fv\,dx$, $v\in H^1_0(\Omega)$, with the positive symmetric bilinear form
\[
B[u,v]:=
\int_{\Omega}\nabla u\cdot \nabla v\,dx
-\int_{\Omega}V(x)uv\,dx.
\]
\end{proof}

\begin{proof}[Proof of \Cref{prop:L1MP}]
Without loss of generality, we may assume that $g=0$. Indeed, let $u_0$ be the solution of
\[\begin{cases}
-\Delta u_0=0
    & \textin \Omega\\
u_0=g
    & \texton \p\Omega.
\end{cases}\]
Then $u_0\geq0$ by the classical maximum principle and $u_0$ is regular in $\Omega$, thus $\tilde{u}=u-u_0$ satisfies
\[\begin{cases}
P\tilde{u}=\tilde{f}
    & \textin \Omega\\
\tilde{u}=0
    & \texton \p\Omega,
\end{cases}\]
for $\tilde{f}=f+Vu_0\in L^1(\Omega;\dist(x,\p\Omega)\,dx)$ and $\tilde{f}\geq0$. Thus we are reduced to the case $g=0$.

The idea is to test with the negative part $u_-=\max\set{-u,0}$, which is supported on $\Omega\setminus B_\delta$ and so $u_-\in L^\infty(\Omega)\subset L^2(\Omega)$.

When $u_-$ is H\"{o}lder continuous, we apply \Cref{lem:H1-exist} to obtain $\zeta\in C^2$ satisfying
\[\begin{cases}
P\zeta=u_-
    & \textin \Omega\\
\zeta=0
    & \texton \p\Omega.
\end{cases}\]
Moreover, $\zeta\geq0$ by the classical maximum principle. Plugging such $\zeta$ into \eqref{eq:L1-sense}, we have
\[
-\int_{\Omega}u_-^2\,dx
=\int_{\Omega}uu_-\,dx
=\int_{\Omega}f\zeta
\geq0.
\]
Thus $u_-\equiv0$, and the proof is completed in the case $u_-$ is smooth enough.

In general, we consider a sequence of mollified negative parts $u_-\ast \eta_{1/k}$ and test the equation with the corresponding $\zeta_k$, which is positive and solves $P\zeta_k=u_-\ast \eta_{1/k}$. We arrive at
\[
\int_{\Omega}u(u_-\ast \eta_{1/k})\,dx \geq0.
\]
Since $u\in L^1(\Omega)$ and $u_-\in L^\infty(\Omega)$, we can take $k\to\infty$ in view of Dominated Convergence Theorem to conclude $u_-\equiv 0$.
\end{proof}

\section*{Acknowledgements}


H.C. has received funding from the European Research Council under the Grant Agreement No 721675. A. DelaTorre is partially supported by Spanish government grant MTM2017-85757-P.

It is a pleasure for H.C. to thank Joaquim Serra for interesting discussions on removable singularities and harmonic capacities, and to thank Mar\'{i}a del Mar Gonz\'{a}lez for a useful discussion on the geometric aspect of the Yamabe problem. It is an honor for H.C. to thank Alessio Figalli for motivating encouragements, and to Juncheng Wei for introducing him the gluing methods. A. DelaTorre is really grateful to M.d.M. Gonz\' alez for introducing her to the Yamabe problem during her PhD thesis and for the useful advice she keeps giving. Both authors would like to thank Ailana Fraser for her remarks, and to Xavier Ros-Oton since this paper is the beginning of a series of works guided to solve a question raised by him during the ``Winter meeting on nonlocal PDEs and applications'' -- UAM Madrid, 2018.  We would also like to thank for the organization of that meeting and the kind hospitality received at the Universidad Aut\'{o}noma de Madrid.

\end{document}